\documentclass{daj}

\usepackage{color,latexsym,amsthm,amsmath,amssymb,path,enumitem}

\newcommand{\ignore}[1]{}

\newtheorem{theorem}{Theorem}[section]
\newtheorem{lemma}[theorem]{Lemma}
\newtheorem{corollary}[theorem]{Corollary}

\newcommand{\Proof}[1]
        {
        \noindent
        \emph{Proof #1.}~
        }
\newsavebox{\smallProofsym}                     % smallproofsym.tex
\savebox{\smallProofsym}                        %
        {%                                      %
        \begin{picture}(7,7)                    %
        \put(0,0){\framebox(6,6){}}             %
        \put(0,2){\framebox(4,4){}}             %
        \end{picture}                           %
        }                                       %
\newcommand{\smalleop}[1]
        {
        \mbox{} \hfill #1~~\usebox{\smallProofsym}\!\!\!\!\!\!\
        }

\usepackage{graphicx,psfrag}

\usepackage{dsfont}

\newcommand{\ZZ}{\ensuremath{\mathbb Z}}
\newcommand{\RR}{\ensuremath{\mathbb R}}

\newcommand{\TT}{\ensuremath{\mathbb T}}

\newcommand{\lattice}{\mathcal P}
\newcommand{\curves}{\Gamma}
\newcommand{\planes}{\Pi}
\newcommand{\spheres}{\mathcal O}
\newcommand{\vrts}{\Pi}

\newcommand{\pts}{\mathcal P}

\def\eps{{\varepsilon}}

\newcommand{\parag}[1]{\vspace{2mm}

\noindent{\bf #1} }

\dajAUTHORdetails{%
  title = {Lower bounds for incidences with hypersurfaces}, %% please capitalize all significant words
  author = {Adam Sheffer},
  plaintextauthor = {Adam Sheffer},
  plaintexttitle = {Lower bounds for incidences with hypersurfaces}, %%  title without math or LaTeX
  keywords = {incidence bounds, hypersurfaces},
}   %%% END \dajAUTHORdetails

\dajEDITORdetails{%
   year={2016},
   number={16},
   received={30 December 2015},   % received date: example: 7 January 2017
   revised={19 July 2016},    % Optional revised date (you may comment it out)
   published={5 October 2016},  % published date
   doi={10.19086/da912},       % XXX = number of paper, e.g. da006 for paper#6
}   %%% END \dajEDITORdetails

\begin{document}

\begin{frontmatter}[classification=text]
%% EDITOR: this will force the keywords to appear right after the Abstract.
%%   If the abstract is too long and would force the keywords off the
%%   front page, please comment out % [classification=text] above
%%   This way the keywords will be floated on the bottom of the first page
%%   even though the Abstract spills over to the next page.

%%% AUTHOR: Title goes here.  This line is optional.  You must use it
%%   if title has footnote attached or requires nontrivial typesetting,
%%   e.g., inclusion of linebreaks to force nice layout.
\title{Lower bounds for incidences with hypersurfaces} %% please capitalize all significant words

%%% AUTHOR:
%%% List all authors. If you wish, place grant acknowledgements in \thanks.
%%% In brackets include a short tag for each author.
\author[pgom]{Adam Sheffer}

%%% AUTHOR: Abstract goes here
\begin{abstract}
We present a technique for deriving lower bounds for incidences with hypersurfaces in $\RR^d$ with $d\ge 4$.
These bounds apply to a variety of hypersurfaces, including hyperplanes, hyperspheres, paraboloids, hypersurfaces of any degree, and sets that are not algebraic.

Beyond being the first non-trivial lower bounds for various incidence problems, our bounds show that some of the known upper bounds for incidence problems in $\RR^d$ are tight up to an extra $\eps$ in the exponent.
Specifically, for every $m$, $d\ge 4$, and $\eps>0$ there exist $m$ points and $n$ hypersurfaces in $\RR^d$ (where $n$ and $m$ satisfy a specific relation) with no $K_{2,d/\eps}$ in the incidence graph and $\Omega\left(m^{(2d-2)/(2d-1)}n^{d/(2d-1)-\eps} \right)$ incidences.
We also provide improved lower bounds for the case of no $K_{s,s}$ in the incidence graph, for a large constant $s$.

Our analysis builds upon ideas from a recent work of Bourgain and Demeter \cite{BD15} on discrete Fourier restriction to the four- and five-dimensional spheres.
Specifically, it is based on studying the additive energy of the integer points in a truncated paraboloid.
\end{abstract}
\end{frontmatter}

%%% AUTHOR: body of paper starts here

%%%%%%%%%%%%%%%%

\section{Introduction}

Given a set $\pts$ of points and a set $\vrts$ of geometric objects (for example, one might consider lines, circles, or hyperplanes) in  $\RR^d$, an \emph{incidence} is a pair $(p,V)\in \pts \times \vrts$ such that the point $p$ is contained in the object $V$.
The number of incidences in $\pts\times\vrts$ is denoted as $I(\pts,\vrts)$.
We sometimes associate the set of incidences with a bipartite graph, called the \emph{incidence graph} of $\pts \times \vrts$.
This graph has vertex sets $\pts$ and $\vrts$, and an edge for every incidence.
In incidence problems, one is usually interested in the maximum number of incidences in $\pts \times \vrts$, taken over all possible sets $\pts,\vrts$ of a given size (equivalently, the maximum number of edges that the incidence graph can contain).
Such incidence bounds have many applications in a variety of fields.\footnote{For a few recent examples, see Guth and Katz's distinct distances result \cite{GK15}, a number theoretic result by Bombieri and Bourgain \cite{Bb15}, and several works in harmonic analysis by Bourgain and Demeter, such as \cite{BD15}.}
The following theorem is one of the main results in incidence theory.

\begin{theorem} \label{th:PS} {\bf (Pach and Sharir \cite{PS92,PS98})}
Let $\pts$ be a set of $m$ points and let $\curves$ be a set of $n$ algebraic curves of degree at most $D$, both in $\RR^2$, such that the incidence graph of $\pts\times\curves$ contains no copy of the complete bipartite graph $K_{s,t}$.
Then
$$I(\pts,\curves) = O_{D,s,t}\left(m^{s/(2s-1)}n^{(2s-2)/(2s-1)}+m+n \right).$$
\end{theorem}

Better bounds are known for some specific types of curves, such as circles and axis-parallel parabolas.
In 2010 Guth and Katz \cite{GK15} introduced the \emph{polynomial partitioning} technique to handle incidences with lines in $\RR^3$.
Since then, this technique has been extended by a series of papers, yielding general bounds for incidences in any dimension.

\begin{theorem} \label{th:UpperBounds}
Let $\pts$ be a set of $m$ points and let $\vrts$ be a set of $n$ varieties of degree at most $D$, both in $\RR^d$, such that the incidence graph of $\pts\times \vrts$ contains no copy of $K_{s,t}$. \\
(a) (Zahl \cite{Zahl13}; see also Kaplan, Matou\v sek, Safernov\'a, and Sharir \cite{KMSS12}) When $d=3$ and every three varieties of $\vrts$ have a zero-dimensional intersection, we have
\[I(\pts,\vrts) = O_{D,s,t}\left(m^{2s/(3s-1)}n^{(3s-3)/(3s-1)}+m+n \right). \]
(b) (Basu and Sombra \cite{BS14}) When $d=4$ and every four varieties of $\vrts$ have a zero-dimensional intersection, we have
\[ I(\pts,\vrts) = O_{D,s,t}\left(m^{3s/(4s-1)}n^{(4s-4)/(4s-1)}+m+n \right). \]
(c) (Fox, Pach, Sheffer, Suk, and Zahl \cite{FPSSZ14}) For any $d\ge 3$ and $\eps>0$, and with no intersection-related restrictions, we have
\[I(\pts,\vrts) = O_{D,s,t}\left(m^{(d-1)s/(ds-1)+\eps}n^{(ds-d)/(ds-1)}+m+n \right). \]
\end{theorem}

Most of the non-trivial lower bounds that are known for incidence problems are for the case of curves (e.g., see Elekes' construction \cite{Elekes01}).
The only non-trivial lower bound that we are aware of for objects of dimension at least two is by Brass and Knauer \cite{BK03}. This is a bound for incidences with hyperplanes with no $K_{s,s}$ in the incidences graph, for a large constant $s$.
The case of large $s$ behaves somewhat differently, and upper bounds better than the ones in Theorem \ref{th:UpperBounds} are known for it (below we present these upper bounds, the lower bound of \cite{BK03}, and an improved lower bound that we derive for this scenario).
It may also be worth mentioning a bound for unit spheres in $\RR^3$ that was observed by Erd\H os \cite{erd60}. This bound is a straightforward extension of the one for the planar unit distances problem \cite{erd46}.

In this paper we provide lower bounds that match the upper bounds of Theorem \ref{th:UpperBounds} (up to an extra $\eps$ in the exponent) in dimension $d\ge 4$, with no $K_{2,t}$ in the incidence graph (for a constant $t$ that depends on $\eps$ and $d$), and with $n$ satisfying a specific relation with $m$.

\begin{theorem} \label{th:LowerBounds}
For any $m$, $d\ge 4$, and $\eps>0$, there exists a set $\pts$ of $m$ points and a set $\planes$ of $n=\Theta\left(m^{(3-3\eps)/(d+1)}\right)$ hyperplanes in $\RR^d$, such that the incidence graph of $\pts\times\planes$ contains no copy of $K_{2,(d-1)/\eps}$ and with $I(\pts,\planes)=\Omega\left(m^{(2d-2)/(2d-1)}n^{d/(2d-1)-\eps} \right)$. This bound remains valid when replacing the hyperplanes with hyperspheres, or with any linearly-closed family of graphs (see definition below).
\end{theorem}
\parag{Remarks.} (a) For the above sizes of $m$ and $n$ we have $m+n=o\left(m^{(2d-2)/(2d-1)}\right)$, so this is indeed the dominating term in the bounds of Theorem \ref{th:UpperBounds}. \\
(b) The extra $\eps$ in the exponents can be removed from the theorem at the cost of replacing the condition of no $K_{2,d/\eps}$ in the incidence graph with no $K_{2,3\lg n}$, and adding a factor of $n^{-c/\lg \lg n}$ to the incidence bound (for some constant $c$);\footnote{Unless stated otherwise, the logarithms in this paper are all with base $e$.} see below for more details.
\vspace{2mm}

We define a \emph{graph} in $\RR^d$ to be a hypersurface that is defined by an equation of the form $x_d = f(x_1,\ldots,x_{d-1})$ (where $f:\RR^{d-1}\to \RR$).
Notice that $f$ is not required to be a polynomial or even algebraic.
We say that a family $F$ of graphs is \emph{linearly-closed} if it satisfies the following property:
If $V\in F$ is a graph that is defined by $x_d = f(x_1,\ldots,x_{d-1})$, then for every $a_1,\ldots,a_d\in \RR$ the graph that is defined by $x_d = f(x_1,\ldots,x_{d-1})+a_1x_1+\cdots + a_{d-1}x_{d-1}+a_d$ is also in $F$.
As examples of linearly-closed families of graphs, consider the set of paraboloids in $\RR^d$ that are defined by $x_d = a_1(x_1-b_1)^2+\cdots+a_d(x_1-b_d)^2$ (where $a_1,\ldots,a_d,b_1,\ldots,b_d\in \RR$), the set of graphs that are also algebraic varieties of degree $k$ in $\RR^d$, and the set of curves in $\RR^2$ that are defined by $y=e^{ax}+b\sqrt{x}+cx+d$ (where $a,b,c,d\in \RR$).

Theorem \ref{th:LowerBounds} is a special case of a more general theorem that we prove.
Specifically, Theorem \ref{th:LowerBounds} is immediately obtained from Theorem \ref{th:GeneralBound} by setting $\delta= \frac{2d-2}{2d-1}$.

\begin{theorem} \label{th:GeneralBound}
Consider $\eps,\delta>0$ and positive integers $m$ and $d\ge 4$.
Then there exist a set $\pts$ of $m$ points and a set $\planes$ of $n=\Theta\left(m^{(3-3\eps)/(d+1)}\right)$ hyperplanes, both in $\RR^d$, such that the incidence graph of $\pts\times \planes$ contains no copy of $K_{2,d/\eps}$, and with
\[ I(\pts,\planes)=\Omega\left(m^{\delta}n^{(d+2-\delta(d+1))/3-\eps}\right). \]
This bound remains valid when replacing the hyperplanes with hyperspheres, or with any linearly-closed family of graphs.
\end{theorem}

\parag{Remarks.} (a) The extra $\eps$ in the exponents can be removed from Theorem \ref{th:GeneralBound} at the cost of replacing the condition of no $K_{2,d/\eps}$ in the incidence graph with no $K_{2,3\lg n}$, and adding a factor of $n^{-c/\lg \lg n}$ to the incidence bound (for some constant $c$). This is immediately obtained by setting $p=1/t$ in the probabilistic argument that is in Section \ref{ssec:Chefnoff}. \\
(b) In Theorem \ref{th:GeneralBound}, our goal was to minimize the value of $t$ in the condition of having no $K_{2,t}$ in the incidence graph (having Theorem \ref{th:LowerBounds} in mind). One can easily change our analysis to obtain bounds for any $t$ between $d/\eps$ and $n^{(d-4)/(d-1)}$ (the case of $t=\Theta\left(n^{(d-4)/(d-1)}\right)$ can be found in Lemma \ref{le:Intermediate}).
\vspace{2mm}

Theorems \ref{th:LowerBounds} and \ref{th:GeneralBound} consider incidences with hypersurfaces, but also apply to surfaces of lower dimensions.
For example, say that we are interested in $d$-dimensional planes in $\RR^{d'}$ (where $d<d'$).
We can choose an arbitrary $(d+1)$-dimensional plane in $\RR^{d'}$ and apply Theorem \ref{th:LowerBounds} or \ref{th:GeneralBound} in it. This might be considered as cheating, since we get a configuration that is not ``truly $d'$-dimensional''.

Our technique is based on a recent work of Bourgain and Demeter \cite{BD15}.
Bourgain and Demeter study discrete Fourier restriction to the four- and five-dimensional spheres, partly by relying on bounds for incidences with hyperplanes in $\RR^4$ and $\RR^5$.
In some sense, we reverse the direction of their analysis to obtain a result about incidences.
Several of the ideas in this paper can be found under disguise in \cite{BD15}.

As already mentioned, for large values of $s$ better bounds are known than the ones in Theorem~\ref{th:UpperBounds}.
\begin{theorem} \label{th:LargeS}
Let $\pts$ be a set of $m$ points and let $\vrts$ be a set of $n$ varieties of degree at most $D$, both in $\RR^d$, such that the incidence graph of $\pts\times\vrts$ contains no copy of $K_{s,s}$ (for some constant $s$). \\
(a) (Brass and Knauer \cite{BK03}; Apfelbaum and Sharir \cite{AS07}) When the elements of $\vrts$ are hyperplanes, we have
\[ I(\pts,\planes) = O\left(m^{d/(d+1)}n^{d/(d+1)} + m + n\right). \]
(b) (Fox, Pach, Sheffer, Suk, and Zahl \cite{FPSSZ14}) For any $d$, $\eps>0$, and set of constant-degree varieties that can be parameterized with $r$ parameters, we have
\[ I(\pts,\planes) = O\left(m^{r(d-1)/(dr-1)+\eps}n^{(r-1)d/(dr-1)} + m + n\right). \]
\end{theorem}

Notice that part (b) of the theorem is a generalization of part (a) (up to the extra $\eps$ in the exponent) since for hyperplanes we have $r=d$. Brass and Knauer \cite{BK03} also presented a configuration of $m$ points and $n$ hyperplanes in $\RR^d$ with no $K_{s,s}$ in the incidence graph, for some large $s$. When $d=3$ this configuration has $\Omega\left((mn)^{7/10}\right)$ incidences, when $d>3$ is odd there are $\Omega\left((mn)^{1-2/(d+3)-\eps}\right)$ incidences, and when $d>3$ is even there are $\Omega\left((mn)^{1-2(d+1)/(d+2)^2-\eps}\right)$ incidences.
Theorem \ref{th:GeneralBound} immediately implies the following improved bound for any $d\ge 4$ (but for $m$ and $n$ that satisfy a specific relation) by setting $\delta=\frac{d+2}{d+4}$.

\begin{corollary} \label{co:LargeS}
Consider an $\eps>0$, and positive integers $m$ and $d\ge 4$.
Then there exist a set $\pts$ of $m$ points and a set $\planes$ of $n=\Theta\left(m^{(3-3\eps)/(d+1)}\right)$ hyperplanes, both in $\RR^d$, such that the incidence graph of $\pts\times \planes$ contains no copy of $K_{2,d/\eps}$, and with $I(\pts,\planes)=\Omega\left((mn)^{1-2/(d+4)-\eps}\right)$.
This bound remains valid when replacing the hyperplanes with hyperspheres, or with any linearly-closed family of graphs.
\end{corollary}

For $d\ge 4$ Corollary \ref{co:LargeS} yields a stronger bound than the one in \cite{BK03}, under a stronger restriction (no $K_{2,d/\eps}$ rather than no $K_{s,s}$ for a large $s$), and for more types of hypersurfaces.
On the other hand, the bound in \cite{BK03} has the advantage of applying to every $m$ and $n$, while Corollary \ref{co:LargeS} requires a specific relation between these two sizes.

The three main tools that are used in the proof of Theorem \ref{th:GeneralBound} (in addition to standard incidence techniques) are bounding the additive energy of a set with the Fourier transform, relying on properties of ellipsoids in high-dimensional lattices, and a probabilistic argument that is based on several Chernoff bounds. In Section \ref{sec:Fourier} we study the additive energy, in Section \ref{sec:lattice} we study ellipsoids and lattices, and in Section \ref{sec:MainProof} we finally prove Theorem \ref{th:GeneralBound}.

In our analysis we rely on ellipsoids that are defined by equations with bounded integer parameters.
In $\RR^4$, we rely on the fact such a two-dimensional ellipsoid can contain many points of of the integer lattice $\ZZ^4$, but a one-dimensional ellipse cannot.
A similar yet more involved argument holds in $\RR^d$ for $d>4$.
For our analysis to hold in $\RR^3$, we would require the existence of such one-dimensional ellipses that contain many points of $\ZZ^3$.
The analysis fails in $\RR^3$ since such ellipses do not exist (for the full details, see Section \ref{sec:lattice}).
Thus, the current state of incidence lower bounds in $\RR^3$ remains somewhat embarrassing.
For example, it seems that for the case of point-plane incidences with no $K_{2,t}$ in the incidence graph, nothing is known beyond the trivial $\Omega(m^{2/3}n^{2/3})$ (this bound is easily obtained from the lower bound for point-line incidences in $\RR^2$).
\vspace{2mm}

\section{Additive energy on the truncated paraboloid} \label{sec:Fourier}

For $d\ge 4$, we denote by $S_d \subset \RR^d$ the paraboloid that is defined by
\[ x_d = x_1^2+\cdots+x_{d-1}^2. \]
Similarly, we define the truncated paraboloid
\begin{equation} \label{eq:truncated}
S_{n,d} = \left\{ (x_1,\ldots,x_d)\in \RR^d : x_d=x_1^2+\cdots+x_{d-1}^2 \text{ and } |x_1|,\ldots,|x_{d-1}| \le n^{\frac{1}{d-1}} \right\}. \end{equation}

Let $\pts = S_{n,d} \cap \ZZ^d$ be the set of points on $S_{n,d}$ with integer coordinates. Notice that $|\pts|=\Theta(n)$.
We are interested in the set of quadruples
\[ Q(\pts) = \{ (a,b,c,d) \in \pts^4 :\, a+b = c+d \}. \]

The quantity $E(\pts)=|Q(\pts)|$ is known as the \emph{additive energy} of $\pts$, and is one of the main objects that are studied  in additive combinatorics (e.g., see \cite[Section 2.3]{TV06}). The proof of Theorem \ref{th:GeneralBound} is based on double counting $|E(\pts)|$.
In the current section we derive the following lower bound for $|E(\pts)|$.
The following lemma is a generalization of Remark 3.2 from \cite{BD15}.

\begin{lemma} \label{le: LowerEnergy}
$E(\pts) = \Omega\left(n^{3-2/(d-1)}\right)$.
\end{lemma}
The idea of using the Fourier transform to bound the additive energy of a set is rather common (e.g., see \cite{SchoenShk13}).
One possible reason for this is that such energy can be expressed using convolutions, and often convolutions are easier to study using the Fourier transform.
In some sense, the paraboloid was chosen since it maximizes the lower bound that is implied by the following argument.
\begin{proof}
Let $e(\xi)=e^{2\pi i\xi}$ and let $\TT^d$ be the $d$-dimensional torus $[0,1)^d$.
We denote by $1_{\pts}$ the indicator function of $\pts$.
That is, for any $x\in \ZZ^d$ we have $1_{\pts}(x)=1$ if $x\in \pts$, and otherwise $1_{\pts}(x)=0$.
We consider the Fourier inversion of  $1_{\pts}$:
\[ f(x) = \sum_{p \in \ZZ^d} 1_{\pts}(p)  e(x\cdot p) = \sum_{p \in \pts}e(x\cdot p). \]
We then have
\begin{align*}
\int_{\TT^d} \left| f(x) \right|^4 &=  \int_{\TT^d}  f(x)^2\overline{ f(x)^2} \\[2mm]
&= \sum_{a_1,a_2,a_3,a_4\in \pts}\int_{\TT^d}e\left((a_1+a_2-a_3-a_4)\cdot x\right).
\end{align*}

Consider a fixed $a_1,a_2,a_3,a_4\in \pts$ with $a_1+a_2-a_3-a_4\neq 0$, and notice that
\[ \int_{x\in \TT^d}e\left((a_1+a_2-a_3-a_4)\cdot x\right) = 0. \]
 On the other hand, when $a_1+a_2-a_3-a_4= 0$ we obviously have
 \[ \int_{x\in \TT^d}e\left((a_1+a_2-a_3-a_4)\cdot x\right) = 1. \]
By combining the two cases, we get that  $\int_{\TT^d} \left| f(x) \right|^4$ is the
number of solutions to $a_1+a_2=a_3+a_4$, taken over all possible $a_1,a_2,a_3,a_4\in \pts$.
In other words,
\[ E(\pts) = \int_{\TT^d} \left| f(x) \right|^4. \]

For a large constant $c$, consider a point $x\in \TT^d$ with $|x_1|,\ldots,|x_{d-1}| <1/(cn^{1/(d-1)})$ and $|x_d|<1/(cn^{2/(d-1)})$. Notice that for any $p\in \pts$ we have $x \cdot p < d/c$. That is, we have a subset of $\TT^d$ of measure $\Omega\left(n^{-(d+1)/(d-1)}\right)$ that contains only points $x$ that satisfy $x\cdot p< d/c$ for every $p\in \pts$.
By taking $c$ to be sufficiently large, we get that $e(x\cdot p)$ is close to 1.
Specifically, $\left| 1-e^{2\pi i x}\right| \le 2\pi \|x\|$, where $\|x\|$ is the distance between $x$ and the closest integer (e.g., see \cite[Section 4.4]{TV06}).
We thus have
\[ E(\pts) = \int_{\TT^d} \left| f(x) \right|^4 =  \int_{\TT^d} \left|\sum_{p \in \pts}e(x\cdot p) \right|^4 = \Omega\left(\frac{1}{n^{(d+1)/(d-1)}} \cdot n^4\right) = \Omega\left(n^{3-2/(d-1)}\right). \]
\vspace{-10mm}

\end{proof}

The bound of Lemma \ref{le: LowerEnergy} is tight for any $d\ge 4$ up to an extra $\eps$ in the exponent (and it seems likely that this extra $\eps$ restriction is unnecessary).
Indeed, combining the analysis in Section \ref{sec:MainProof} with the bound in part (c) of Theorem \ref{th:UpperBounds} yields $E(\pts) = O\left(n^{3-2/(d-1)+\eps}\right)$.

\section{Ellipsoids containing lattice points} \label{sec:lattice}

In this section we derive upper bounds for the maximum number of integer lattice points that various types of ellipsoids in $\RR^d$ can be incident to.
For this purpose, we rely on the following result of Dirichlet (for example, see \cite[Section 11.2]{Iwaniec97}).
In this section of our paper (but not in the other ones), when writing $\left(\frac{a}{b}\right)$ we refer to the Kronecker symbol rather than to division.

\begin{theorem} \label{th:BinaryQuad}
Let $a,b,c,n$ be positive integers such that $\gcd(a,b,c)=1$, and set $D=b^2-4ac$.
If $D<0$ then there exists an integer $2\le k \le 6$ such the number of integer solutions to $ax^2+bxy+cy^2=n$ is
\[ k\sum_{d|n} \left(\frac{D}{d}\right). \]
\end{theorem}

We first consider the case of ellipses in $\RR^d$ for $d\ge 4$. We say that such an ellipse is $n$-\emph{proper} if it is the intersection of the paraboloid $S_d$ and $d-2$ hyperplanes with the following properties.
Two of the hyperplanes are defined by linear equations with every coefficient being an integer with an absolute value of size $O_d(n)$.
Each of the remaining $d-4$ hyperplanes is orthogonal to a different axis.
Notice that an intersection of $S_d$ with $d-2$ hyperplanes may not result in an ellipse.
In the following we only consider intersections that do form ellipses.

\begin{lemma} \label{le:EllipseLattice}
For $d\ge 4$, let $\lattice$ be a section of the $d$-dimensional integer lattice of size $n\times n \times \cdots \times n$ (so $\lattice$ consists of $n^d$ points of $\ZZ^d$).
Let $\gamma$ be an $n$-proper ellipse. Then $\gamma$ contains $O_d\left(n^{c_d/\lg \lg n}\right)$ points of $\lattice$ (for some constant $c_d$ that only depends on $d$).
\end{lemma}
\begin{proof}
We project both $\lattice$ and $\gamma$ onto the two-dimensional plane $\Pi$ that is spanned by the first two axes.
This yields an $n\times n$ lattice $\lattice'$ and an ellipse $\gamma'$.
In the rest of the proof we ignore $\RR^d$ and work only in $\Pi$.

We can obtain the equation that defines $\gamma'$ by eliminating $x_3,\ldots,x_d$ in the equation $x_d=x_1^2+\cdots+x_{d-1}^2$, using the $d-2$ linear equations that define the hyperplanes in the definition of $\gamma$.\footnote{Obtaining a projection by eliminating variables gives a variety that contains the projection, but is not necessarily identical to it (e.g., see \cite[Chapter 3]{CLOu}). This is sufficient for our purpose, since we are interested in an upper bound for the number of incidences, and this step can only increase their number (moreover, in the special case of an $n$-proper ellipse the projection is equivalent to the variety).} That is, $\gamma$ is the zero set of an equation of the form
\[ a_1x^2+a_2y^2+a_3xy+a_4x+a_5y-a_6, \]
where $a_1,\ldots,a_6$ are rational numbers with numerators and denominators of size at most $n^{O_d(1)}$.

We next study the center $p=(p_x,p_y)$ of $\gamma'$.
We translate the plane by $-p$ so that $\gamma'$ is centered at the origin.
The equation that defines $\gamma'$ becomes
\[ a_1(x+p_x)^2+a_2(y+p_y)^2+a_3(x+p_x)(y+p_y)+a_4(x+p_x)+a_5(y+p_y)=-a_6. \]

Since an ellipse that is centered at the origin has no linear terms in its defining equation, we have
\begin{align} \label{eq:TwoLinear}
2a_1p_x+a_3p_y+a_4=0 \qquad \text{ and } \qquad 2a_2 p_y+a_3p_x+a_5=0.
\end{align}

Since $\gamma'$ is an ellipse by assumption, we have $a_1,a_2\neq 0$ and $a_3^2<4a_1a_2$ ($a_3^2=4a_1a_2$ corresponds to a parabola and $a_3^2>4a_1a_2$ corresponds to a hyperbola).
These properties imply that there there exists a unique solution to \eqref{eq:TwoLinear} (when considering $p_x$ and $p_y$ as the variables).
In this solution $|p_x|$ and $|p_y|$ are rational numbers with denominators and numerators of size at most $n^{O_1(d)}$.
That means that we can refine $\lattice'$ to be an $n'\times n'$ square lattice $\lattice''$ that contains $p$, with $n' = n^{O_1(d)}$ (that is, the refined lattice is defined by two vectors of the same size and in the directions of the axes, and it fully contains $\lattice'$).

We perform a uniform scaling so that $\lattice''$ becomes an $n'\times n'$ section $\lattice^*$ of the integer lattice.
The translation and scaling take $\gamma'$ into an ellipse $\gamma^*$ that is the zero set of
\[ b_1x^2+b_2y^2+b_3xy-b_4, \]
where $b_1,\ldots,b_4$ are rational numbers with numerators and denominators of size at most $n^{O_d(1)}$.

To remove the denominators, we multiply this polynomial by the smallest common multiple of the denominators.
We then find the greatest common divisor $g$ of the resulting integer coefficients and divide the polynomial by $g$.
We get that $\gamma^*$ is the zero set of
\[ c_1x^2+c_2y^2+c_3xy-c_4, \]
where $c_1,\ldots,c_4$ are integers of size at most $n^{O_d(1)}$ with $\gcd(c_1,c_2,c_3,c_4)=1$.

If $g'=\gcd(c_1,c_2,c_3)>1$, then there are no integer solutions to $c_1x^2+c_2y^2+c_3xy=c_4$ since the above implies that $g'$ does not divide $c_4$. That is, we may assume that $\gcd(c_1,c_2,c_3)= 1$.
Since $\gamma^*$ is still an ellipse, we also have $c_3^2<4c_1c_2$.
Thus, we may apply Theorem \ref{th:BinaryQuad} to $c_1x^2+c_2y^2+c_3xy=c_4$.
This implies that the number of points of $\ZZ^2$ that are contained in $\gamma^*$ is $O\left(\sum_{m|c_4} \left(\frac{c_3^2-4c_1c_2}{m}\right)\right)$.
As an upper bound for this expression, we consider the case where  every element in the sum equals 1. That is, the number of solutions is at most linear in the number of divisors of $c_4$.
This number of divisors is at most $c_4^{O(1/\lg\lg c_4)}$ (e.g., see \cite[Section 1.6]{Tao09}).
Thus, the ellipse $\gamma^*$ contains $O_d\left(n^{c_d/\lg \lg n}\right)$ points of $\ZZ^2$ (for some constant $c_d$ that depends on $d$).
That is, $\gamma^*$ contains $O_d\left(n^{c_d/\lg \lg n}\right)$ points of $\lattice^*$, which in turn implies that $\gamma$ contains $O_d\left(n^{c_d/\lg \lg n}\right)$ points of $\lattice$.
\end{proof}

We now extend the result of Lemma \ref{le:EllipseLattice} to ellipsoids of larger dimensions. Let $1\le k\le d-3$. We say that a $k$-dimensional ellipsoid in $\RR^d$ is $n$-\emph{proper} if it is the intersection of the paraboloid $S_d$ and $d-k-1$ hyperplanes with the following properties.
Two of the hyperplanes are defined by linear equations with every coefficient being an integer with an absolute value of size $O_d(n)$.
Each of the remaining $d-4$ hyperplanes is orthogonal to a different axis.

\begin{lemma} \label{le:ellipsoids}
For $d\ge 4$, let $\lattice$ be a section of the $d$-dimensional integer lattice of size $n\times n \times \cdots \times n$ (so $\lattice$ consists of $n^{d}$ points of $\ZZ^d$) and let $1\le k\le d-3$.
Let $E$ be an $n$-proper $k$-dimensional ellipsoid, also in $\RR^d$.
Then $E$ contains $O_d\left(n^{k-1+c_d/\lg \lg n}\right)$ points of $\lattice$ (for some constant $c_d$ that only depends on $d$).
\end{lemma}
\begin{proof}
The proof is by induction on $k$. For the induction basis, the case of $k=1$ is Lemma \ref{le:EllipseLattice}.
For the induction step, consider a $k$-dimensional $n$-proper ellipsoid $E$ in $\RR^d$ (where $1<k\le d-3$).
Let $h$ denote the $(k+1)$-dimensional flat that fully contains $E$, and notice that $\lattice' = h\cap \lattice$ is a $(k+1)$-dimensional rectangular lattice of size at most $n\times n \times \cdots \times n$ (if $\lattice'$ is of a lower-dimension, then we consider a smaller-dimensional flat $h'$ that contains $\lattice'$ and use the induction hypothesis on $h'\cap S$).
We can cover the points of $\lattice'$ with at most $n$ parallel $k$-dimensional flats that are fully contained in $h$, with each such flat containing a $k$-dimensional rectangular lattice of size at most $n\times n \times \cdots \times n$.
Moreover, we can take these $k$-dimensional flats to be orthogonal to one of the axes.
For each of these flats $h_i$, notice that $h_i \cap E$ is an $n$-proper ellipsoid of dimension at most $k-1$ (or an empty set).
By the induction hypothesis, each such lower-dimensional ellipsoid contains $O_d\left(n^{k-2+c/\lg \lg n}\right)$ points. The assertion of the lemma is obtained by summing up this bound over the lower-dimensional ellipsoids.
\end{proof}

It seems likely that the bound of Lemma \ref{le:ellipsoids} can be slightly improved to $O_d\left(n^{k-1}\right)$.
This is known to be the case for $k$-dimensional spheres that are centered at a point of the lattice (e.g., see \cite{Gross21}).
Bombieri and Pila \cite{BP89} proved that any irreducible algebraic curve of degree $k$ in $\RR^2$ contains $O_k(n^{1/(2k)})$ points of a $\sqrt{n}\times\sqrt{n}$ section of the integer lattice.
While it is easy to verify that this bound is tight, it seems possible that a much stronger bound holds for curves such as arbitrary circles and ellipses.
We are not aware of any results in this direction.
We are also not aware of any non-trivial bounds for general ellipsoids in $\RR^d$.

\section{Proof of Theorem \ref{th:LowerBounds}} \label{sec:MainProof}

This section contains the main part of our analysis, and it is divided into two parts. In the first part we prove the following lemma, and in the second we use this lemma to derive Theorem \ref{th:LowerBounds}.

\begin{lemma} \label{le:Intermediate}
Consider constants $\beta>2$ and $\gamma\ge 0$, positive integers $n$ and $d\ge 4$, and let $\alpha = \frac{\beta d+ d-3\beta+1}{d-1}$.
Then there exists a set $\pts$ of $m=\Theta\left(n^{(d+1)/(d-1)}\right)$ points and a set $\planes$ of $n$ hyperplanes, both in $\RR^d$, such that the incidence graph of $\pts\times \planes$ contains no copy of $K_{2,t}$ where $t=\Theta\left(n^{(d-4)/(d-1)+ c/\lg \lg n}\right)$ (for some constant $c$), and with
\[ I(\pts,\planes)=\Omega\left(m^{(\beta-1)/\beta} n^{\left(\alpha-\gamma\frac{d-4}{d-1}-\frac{c}{\lg \lg n}\right)/\beta}t^{\gamma/\beta}\right). \]
This bound remains valid when replacing the hyperplanes with hyperspheres, or with any linearly-closed family of graphs.
\end{lemma}

\subsection{Proving Lemma \ref{le:Intermediate}} \label{ssec:InterLemma}

We recall the definitions of $S_d$ and $S_{n,d}$ from Section \ref{sec:Fourier}.
We again define $\pts = S_{n,d} \cap \ZZ^d$ and recall that $|\pts|=\Theta(n)$.
By Lemma \ref{le: LowerEnergy} we have
\begin{equation} \label{eq:Lower4}
E(\pts)=\Omega\left(n^{3-2/(d-1)}\right).
\end{equation}

We will now reduce the problem of obtaining an upper bound for $E(\pts)$ to a point-hyperplane incidence problem.
The existence of various upper bounds for the maximum number of point-hyperplane incidences in $\RR^d$ would then imply an upper bound on $E(\pts)$ that contradicts \eqref{eq:Lower4}.
That is, we would get a contradiction to the existence of various point-hyperplane incidence upper bounds.
Such a contradiction implies that there exist point-hyperplane configurations with a large number of incidences (the full proof below is more constructive and also provides some information about how these configurations look like).

Consider points $a,b\in \pts$ such that $a+b=v\in\RR^d$.
We have
\begin{align*}
v_d &= a_1^2+\cdots+a_{d-1}^2 + b_1^2+\cdots+b_{d-1}^2 \\
        &= a_1^2+\cdots+a_{d-1}^2 + (v_1-a_1)^2+\cdots+(v_{d-1}-a_{d-1})^2 \\
        &= 2a_d -2a_1v_1-\cdots-2a_{d-1}v_{d-1} +\left(v_1^2+\cdots+v_{d-1}^2\right).
\end{align*}
That is, if $a,b\in \pts$ satisfy $a+b=v$ then $a$ and $b$ are both contained in the hyperplane $H_v$ that is defined by
\begin{equation} \label{eq:PiPlane}
2x_d -2x_1v_1-\cdots-2x_{d-1}v_{d-1} + \left(v_1^2+\cdots+v_{d-1}^2\right)-v_d=0.
\end{equation}

Let $\planes = \{H_v :\, v=a+b \text{ with } a,b\in \pts \}$.
%Since every pair of points of $\pts$ corresponds to at most one hyperplane of $\planes$, we have the trivial bound $|\planes| = O(n^2)$.
We denote by $N_k$ the number of planes $H_v\in \planes$ that contain between $2^{k}$ and $2^{k+1}-1$ pairs $a,b\in \pts$ with $a+b=v$.
Since every pair of points of $\pts$ corresponds to a unique $v$, we have $\sum_k N_k \cdot 2^k = O(n^2)$.
For an $r$ that we will determine below, we have
\begin{equation} \label{eq:ESn4}
E(\pts) < \sum_{k=0}^{\lg n} N_k 2^{2(k+1)} = \sum_{k=0}^{r} N_k 2^{2(k+1)} + \sum_{k=r+1}^{\lg n} N_k 2^{2(k+1)}.
\end{equation}

By the above bound for $\sum_k N_k \cdot 2^k$, we have $\sum_{k=0}^{r} N_k 2^{2(k+1)} = O\left(n^2 2^r\right)$.
To handle the second sum, we notice that for any $a\in H_v$ there exists at most one $b\in \pts$ such that $a+b=v$.
Thus, the number of pairs in $H_v$ is at most $|H_v \cap \pts|/2$.
Assume that we have the bound
\begin{equation} \label{eq:NkHyp4}
N_k = O\left(\frac{n^\alpha}{2^{\beta k}}+\frac{n}{2^k}\right),
\end{equation}
for some $\beta> 2$ and $\alpha = (\beta d+ d-3\beta+1)/(d-1)$ (note that for any $\beta>2$ we have $\alpha>0$).
We then get that
\[ \sum_{k=r+1}^{\lg n} N_k 2^{2(k+1)} = \sum_{k=r+1}^{\lg n} O\left(\frac{n^{\alpha}}{2^{(\beta-2)k}}+n2^k\right) = O\left(\frac{n^{\alpha}}{2^{(\beta-2)r}}+n^2\right).\]

To optimize the bound in \eqref{eq:ESn4}, we need $r$ to satisfy $n^2 2^r = \Theta\left(n^{\alpha}/2^{(\beta-2)r}\right)$ (the term $n^2$ in the above bound is always subsumed by $n^2 2^r$).
Since $\beta>2$, we have that $n^{\alpha}/2^{(\beta-2)r}$ is decreasing in $r$, and thus the optimal value for $r$ is indeed obtained when the two bounds are equal.
That is, we set $r=\frac{2-\alpha}{1-\beta}\lg_2 n = \frac{d-3}{d-1}\lg_2 n$ and obtain
\begin{equation} \label{eq:EPupper}
E(\pts) < \sum_{k=0}^{r} N_k 2^{2(k+1)} + \sum_{k=r+1}^{\lg n} N_k 2^{2(k+1)} = O\left(n^{2+(d-3)/(d-1)}\right) = O\left(n^{3-2/(d-1)}\right).
\end{equation}

Notice that this upper bound for $E(\pts)$ matches the lower bound in \eqref{eq:Lower4}.
If no $k'$ satisfies $N_{k'} = \Omega\left(n^\alpha/2^{\beta k'}\right)$, we would be able to improve the above upper bound, which would contradict \eqref{eq:Lower4}.
Thus, such a $k'$ must exist.
Moreover, there must exist such a $k'$ with $|k'-r|=O(1)$.
Indeed, assume for contradiction that there exist functions $f_1(n)=\omega(1)$ and $f_2(n)=\omega(1)$ such that for every $k$ that satisfies $|k-r|\le f_1(n)$ we have $N_{k} = O\left(n^\alpha/(2^{\beta k}f_2(n))\right)$.
We set $f(n)= \min\{f_1(n),f_2(n)\}$, $r'=r-(2\beta-4)^{-1}\cdot \lg_2 f(n)$, and $r''= r+(2\beta-4)^{-1}\lg_2 f(n)$.
 By repeating the calculation in \eqref{eq:EPupper}, we obtain
\begin{align*}
E(\pts) &< \sum_{k=0}^{r'-1} N_k 2^{2(k+1)} + \sum_{k=r'}^{r''-1} N_k 2^{2(k+1)} +\sum_{k=r''}^{\lg n} N_k 2^{2(k+1)} \\
&= O\left(n^2 2^{r'} + \frac{n^\alpha}{2^{(\beta-2) r'}f(n)} + \frac{n^\alpha}{2^{(\beta-2) r''}}\right) \\
&= O\left(\frac{n^{3-2/(d-1)}}{f(n)^{1/(2\beta-4)}}+ \frac{n^{3-2/(d-1)}}{f(n)^{1/2}}+ \frac{n^{3-2/(d-1)}}{f(n)^{1/2}}\right) = o\left(n^{3-2/(d-1)}\right).
\end{align*}
This contradicts \eqref{eq:Lower4}, so there must exist $k'$ that satisfies $|k'-r|=O(1)$ and $N_{k'} = \Omega\left(n^\alpha/2^{\beta k'}\right)$.

Let $\planes'$ denote the set of hyperplanes $H_v \in \planes$ that contain between $2^{k'}$ and $2^{k'+1}-1$ pairs $a,b\in \pts$ with $a+b=v$. Notice that $N_{k'} = |\planes'|$.
Every hyperplane of $\planes'$ intersects $S_d$ in a $(d-2)$-dimensional ellipsoid.
(While there exist hyperplanes that intersect $S_d$ in a paraboloid, none of these are in $\planes'$. Indeed, notice that the coefficient of $x_d$ in \eqref{eq:PiPlane} is constant.)
Similarly, two hyperplanes of $\planes'$ intersect $S_{d}$ in an ellipsoid of dimension at most $d-3$.
By inspecting the defining equation \eqref{eq:PiPlane} of the hyperplanes of $\planes'$, we notice that all of the coefficients in it are integers of absolute value $O_d(n^{2/(d-1)})$.
Thus, the ellipsoids that are obtained from such intersections are $n$-proper.
By Lemma \ref{le:ellipsoids}, the intersection of two hyperplanes of $\planes'$ cannot contain $t=\Theta\left(n^{(d-4)/(d-1)+ c/\lg \lg n}\right)$ points of $\pts$.
In other words, the incidence graph of $\pts \times \planes'$ contains no copy of $K_{t,2}$.

We use the standard point-hyperplane duality on $\pts$ and $\planes'$ to obtain a configuration with no $K_{2,t}$ in the incidence graph. We denote the new point set as $\planes^*$ and the new set of hyperplanes as $\pts^*$.
With this new notation, we have $|\pts^*|=\Theta(n)$, $|\planes^*|= \Omega_d\left(n^\alpha/2^{\beta k'}\right)$, and that every point of $\planes^*$ is incident to at least $2^{k'}$ hyperplanes of $\pts^*$.
We arbitrarily remove points from $\planes^*$ to  obtain $|\planes^*|= \Theta_d\left(n^\alpha/2^{\beta k'}\right)$. This implies
\begin{align*}
I(\planes^*,\pts^*) &= \Omega\left(|\planes^*| \cdot 2^{k'}\right) = \Omega_d\left(|\planes^*|^{(\beta-1)/\beta} \left(n^\alpha/2^{\beta k'}\right)^{1/\beta}\cdot 2^{k'}\right) \\
&= \Omega_d\left(|\planes^*|^{(\beta-1)/\beta} |\pts^*|^{\alpha/\beta}\right) = \Omega_d\left(|\planes^*|^{(\beta-1)/\beta} |\pts^*|^{\left(\alpha-\gamma\frac{d-4}{d-1}-\frac{\gamma c}{\lg \lg n}\right)/\beta}t^{\gamma/\beta}\right).
\end{align*}

It remains to bound the size of $|\planes^*|$. Recall  that $|\planes^*|= \Theta\left(n^\alpha/2^{\beta k'}\right)$.
Moreover, we know that $2^{k'} = \Theta\left(2^r\right) = \Theta\left(n^{(d-3)/(d-1)}\right)$. We conclude that
\[ |\planes^*| = \Theta\left(n^\alpha/n^{\beta(d-3)(d-1)}\right) = \Theta\left(n^{\frac{\beta d+ d-3\beta+1}{d-1}-\frac{\beta d-3\beta}{d-1}}\right)=\Theta\left(n^{(d+1)/(d-1)}\right). \]
This completes the case of incidences with hyperplanes.

\parag{Other hypersurfaces}
We now show how to apply the above argument to other types of hypersurfaces.
The case of hyperspheres is obtained by a simple use of the \emph{inversion transformation} around the origin $\rho_d:\RR^d \to \RR^d$ (e.g., the planar case can be found in \cite[Section 37]{Hart00}).  The transformation $\rho_d(\cdot)$ maps the point
$p=(x_1,\ldots,x_d)\neq (0,\ldots,0)$ to the point $\rho_d(p)=(\bar{x}_1,\ldots,\bar{x}_d)$,
where
\[ \bar{x}_i = \frac{x_i}{x_1^2+\cdots+x_d^2},\quad i=1,\ldots,d. \]

If $h\subset \RR^d$ is a hyperplane that is not incident to the origin then $\rho_d(h)$ is a hypersphere that is incident to the origin.
This property is easy to verify after noticing that $\rho_d(\cdot)$ is its own inverse.
Another important observation is that a point $p\in \RR^d \setminus \{0\}$ is incident to an object $V \subset \RR^d$ if and only if $\rho_d(p)$ is incident to $\rho_d(V)$.

We consider the sets $\planes^*$ and $\pts^*$ that were obtained above, and perform a translation of $\RR^d$ so that no plane of $\pts^*$ is incident to the origin. We then set $\pts'' = \rho_d(\planes^*)$ and $\spheres = \rho_d(\pts^*)$, and notice that $\spheres$ is a set of hyperspheres. As before, we have $|\spheres|=\Theta(n)$, $|\pts''| = \Theta\left(n^{(d+1)/(d-1)}\right)$, no $K_{2,t}$ in the incidence graph of $\pts''\times \spheres$, and
\[ I(\pts'',\spheres) =\Omega\left(|\pts'|^{(\beta-1)/\beta} |\planes'|^{\left(\alpha-\gamma\frac{d-4}{d-1}-\frac{\gamma c}{\lg \lg n}\right)/\beta}t^{\gamma/\beta}\right). \]

We next consider the case of a linearly-closed family $F$ of graphs.
Consider a function $f:\RR^{d-1}\to \RR$ such that $x_d= f(x_1,\ldots,x_{d-1})$ defines a graph of $F$.
We rely on a technique that was introduced by J\'ozsef Solymosi (and should appear in \cite{SolySza}).
We begin with a point-hyperplane configuration in $\RR^d$, as described above.
We consider the bijection $\phi(x_1,x_2,\ldots,x_d) =(x_1,\ldots,x_{d-1},x_d+f(x_1,\ldots,x_{d-1}))$ from $\RR^d$ to itself, and apply $\phi$ on the above point-hyperplane configuration.
Notice that $\phi$ is a bijection since $\phi^{-1}(x_1,x_2,\ldots,x_d) =(x_1,\ldots,x_{d-1},x_d-f(x_1,\ldots,x_{d-1}))$.
Thus, $\phi$ maintains the incidence structure of the point-hyperplane configuration.

Consider a hyperplane $H$ that is defined by the equation $x_d = a_1x_1+\cdots + a_{d-1}x_{d-1}$ (all of the hyperplanes in our point-hyperplane configuration can be defined like this), and notice that $\phi(H)$ is defined by $x_d = a_1x_1+\cdots + a_{d-1}x_{d-1} +f(x_1,\ldots,x_{d-1})$.
Since $F$ is linearly-closed, $\phi(H)\in F$.
That is, we obtain a configuration of points and graphs from $F$ with the same number of incidences and no $K_{2,t}$ in the incidence graph.

\subsection{Using Lemma \ref{le:Intermediate}} \label{ssec:Chefnoff}

For our analysis we require several Chernoff bounds (e.g., see \cite{CL06} and the last equation of \cite{HR90}).

\begin{lemma}[Chernoff bounds] \label{le:Chernoff}
(a) Given $0 < p < 1$, for every $1\le i \le n$ let $X_i$ be a random variable that equals 1 with probability $p$ and otherwise equals 0. Moreover, let these $n$ variables be pairwise independent.
Let $X= \sum_{i=1}^n X_i$, so $E[X] = \sum_{i=1}^n E[X_i] = pn$.
Let $k$ be an integer larger than $pn$. Then for any $\lambda,k>0$
\begin{align*}
\Pr[X\ge pn+\lambda] &\le e^{\frac{-\lambda^2}{2pn+2\lambda/3}}, \\[2mm]
\Pr[X\le pn-\lambda] &\le e^{-\lambda^2/2pn}, \\[2mm]
\Pr[X\ge k] &\le \left(\frac{np}{k}\right)^k e^{k-np}.
\end{align*}

(b) Given $0 < p < 1$, for every $1\le i \le n$ let $X_i$ be a random variable that equals $m_i$ with probability $p$ and otherwise equals 0. Moreover, let these $n$ variables be pairwise independent.
Let $X= \sum_{i=1}^n X_i$, so $E[X] = \sum_{i=1}^n E[X_i] = p\sum_{i=1}^n m_i$. Then for any $\lambda>0$
\begin{align*}
\Pr[X\le E[X]-\lambda] &\le e^{\frac{-\lambda^2}{2\sum_{i=1}^n E[X_i^2]}}.
\end{align*}
\end{lemma}

For simplicity we consider a set $\planes$ of hyperplanes.
The following analysis remains valid after replacing the hyperplanes with a different type of objects.

From Lemma \ref{le:Intermediate} we have a set $\pts$ of $m=\Theta\left(n^{(d+1)/(d-1)}\right)$ points and a set $\planes$ of $a \cdot n$ hyperplanes (for some constant $a$), both in $\RR^d$, such that the incidence graph of $\pts\times \planes$ contains no copy of $K_{2,t}$ with $t=\Theta\left(n^{(d-4)/(d-1)+c/\lg\lg n}\right)$, and with
\[ I(\pts,\planes)=\Omega\left(m^{(\beta-1)/\beta} n^{\left(\alpha-\gamma\frac{d-4}{d-1}-\frac{\gamma c}{\lg\lg n}\right)/\beta}t^{\gamma/\beta}\right). \]

Set $p=1/(t n^\eps)$ and let $\vrts'$ be a set that is obtained by taking any element of $\vrts$ with probability $p$.
By lemma \ref{le:Chernoff}, we have\begin{itemize}
\item The probability of $|\vrts'|\le anp/10$ is at most $e^{-0.9^2a^2p^2n^2/2apn}<e^{-0.4apn}$.
\item The probability of $|\vrts'|\ge 10anp$ is at most $e^{\frac{-81a^2p^2n^2}{2pn+6apn}}<e^{-10pn\cdot \min\{a,a^2\}}$.
 \item The probability that a given pair of points of $\pts$ is fully contained in at least $3/\eps$ hyperplanes of $\vrts'$ is smaller than
     \[ \left(\frac{tp}{3/\eps}\right)^{3/\eps} e^{3/\eps-tp}<\left(\frac{\eps}{3n^\eps}\right)^{3/\eps}e^{3/\eps}<\frac{e^{3/\eps}}{n^3}. \]
\item The probability that no pair of points of $\pts$ is fully contained in $3/\eps$ hyperplanes of $\vrts'$ is smaller than $n^2 \frac{e^{3/\eps}}{n^3} < \frac{e^{3/\eps}}{n}$.
\item The probability that $I(\pts,\vrts') \le I(\pts,\vrts)p/100$ is at most
\begin{align*} e^{\frac{-(99\cdot I(\pts,\vrts)p/100)^2}{2\cdot I(\pts,\vrts)mp}} &< e^{\frac{-0.49 p \cdot I(\pts,\vrts)}{m}} = e^{-0.49 n^{\frac{\alpha}{\beta}-\eps-\frac{d+1}{(d-1)\beta}-\frac{d-4}{d-1}-\frac{c}{\lg\lg n}}} \\
& = e^{-0.49 n^{\frac{1}{(d-1)}-\eps-\frac{c}{\lg\lg n}}} < e^{-0.49 n^{\frac{1}{d-1}-2\eps}}.
\end{align*}
 This bound is obtained from part (b) of Lemma \ref{le:Chernoff}. By taking $m_i$ to be the number of incidences on the $i$'th hyperplane of $\vrts$ we have $m_i \le m$, so $\sum_{i=1}^n E[X_i^2] < I(\pts,\vrts)mp$. We also assume that $n$ is sufficiently large so that $c/\lg\lg n <\eps$.
 \end{itemize}

By taking sufficiently large $n$, all of the above probabilities are smaller than $10^{-10}$. That is, with probability of at least $1-4\cdot 10^{-10}$ we have that $|\vrts'|=\Theta(np)$, that no pair of points of $\pts$ is fully contained in $3/\eps$ elements of $\vrts'$, and that $I(\pts,\vrts') > I(\pts,\vrts')p/100$. Thus, there must exist a set $\vrts'$ that satisfies all of these properties, and we consider such a set. Notice that the incidence graph of $\pts\times\vrts'$ contains no copy of $K_{2,3/\eps}$, that $|\vrts'| = \Theta\left(n^{3/(d-1)-\eps}\right)$, and that
\begin{align*}  I(\pts,\vrts') &> \frac{p}{100} \cdot I(\pts,\vrts) =\Omega\left(\frac{\left(n^{(d+1)/(d-1)}\right)^{(\beta-1)/\beta} n^{\alpha/\beta}}{n^{(d-4)/(d-1)+\eps}}\right) \\[2mm]
& = \Omega\left(n^{\frac{d+1}{d-1}\frac{\beta-1}{\beta}+\frac{\alpha}{\beta}-\frac{d-4}{d-1}-\eps}\right) = \Omega\left(n^{\frac{(d+1)(\beta-1)+(\beta d+ d-3\beta+1)-\beta(d-4)}{\beta(d-1)}-\eps}\right) \\[2mm]
&=\Omega\left(n^{(d+2)/(d-1)-\eps}\right) = \Omega\left(|\pts|^{\delta}|\vrts'|^{\left(\frac{d+2}{d-1}-\frac{\delta(d+1)}{d-1}-\eps\right)\frac{d-1}{3}}\right) \\[2mm]
&= \Omega\left(|\pts|^{\delta}|\vrts'|^{(d+2-\delta(d+1)-\eps(d-1))/3}\right).
\end{align*}

Note that $|\planes'| = \Theta\left(n^{3/(d-1)-\eps-c/\lg\lg n}\right)= \Theta\left(m^{\frac{3-(\eps-c/\lg\lg m) (d-1)}{d+1}}\right)$. Replacing $\eps$ with $(d-1)\eps/3$ yields the assertion of Theorem \ref{th:GeneralBound} (the term $c/ \lg\lg m$ is negligible with respect to $\eps$, so we ignore it).

%%% AUTHOR: optional acknowledgments here
\section*{Acknowledgments} %%  you may comment this out if no Ackno
The author is indebted to Ciprian Demeter, who provided the initial inspiration for this work.
He would like to thank Prabath Silva, J\'ozsef Solymosi, Joshua Zahl, and Frank de Zeeuw for several helpful discussions.
He would also like to thank the anonymous referees for carefully reading the paper and helping to improve it.

%%% AUTHOR:
%%% Bibliography goes here. Note that the arXiv cannot process bibtex
%%% or biber bibliographies.  Example of acceptable bibliograpy format:
\bibliographystyle{amsplain}

\begin{thebibliography}{99}
%
\bibitem{AS07}
R.\ Apfelbaum and M.\ Sharir.
\newblock Large bipartite graphs in incidence graphs of points and hyperplanes.
\newblock \emph{SIAM J.\ Discrete Math.} {\bf 21} (2007), 707--725.
%

\bibitem{BS14}
S.\ Basu and M.\ Sombra.
\newblock Polynomial partitioning on varieties and point-hypersurface incidences in four dimensions.
\newblock \emph{Discrete Comput.\ Geom.} {\bf 55} (2016), 158--184.
%

\bibitem{Bb15}
E.\ Bombieri and J.\ Bourgain.
\newblock A problem on sums of two squares.
\newblock \emph{Int.\ Math.\ Res.\ Not.} (2015), 3343--3407.
%

\bibitem{BP89}
E.\ Bombieri and J.\ Pila.
\newblock The number of integral points on arcs and ovals.
\newblock \emph{Duke Math.\ J.}  {\bf 59} (1989), 337--357.
%

\bibitem{BD15}
J.\ Bourgain and C.\ Demeter.
\newblock New Bounds for the Discrete Fourier Restriction to the Sphere in 4D and 5D.
\newblock \emph{Int.\ Math.\ Res.\ Not.} (2015), 3150--3184.
%

\bibitem{BK03}
P.\ Brass and C.\ Knauer.
\newblock On counting point-hyperplane incidences.
\newblock \emph{Computational Geometry} {\bf 25} (2003), 13--20.
%

\bibitem{CL06}
F.\ Chung and L.\ Lu.
\newblock Concentration inequalities and martingale inequalities: a survey.
\newblock \emph{Internet Mathematics} {\bf 3} (2006), 79--127.
%

\bibitem{CLOu}
D.\ Cox, J.\ Little, and D.\ O'Shea.
\newblock {\it Ideals, Varieties, and Algorithms: An Introduction to
Computational Algebraic Geometry and Commutative Algebra}, 3rd edition,
Springer-Verlag, Heidelberg, 2007.
%

\bibitem{Elekes01}
G.\ Elekes.
\newblock Sums versus products in number theory, algebra and Erd\H os geometry.
\newblock in \emph{Paul Erd\H os and his Mathematics II}, 2001, 241--290.
%

\bibitem{erd46}
P.\ Erd\H os.
\newblock On sets of distances of $n$ points.
\newblock \emph{Amer. Math. Monthly} {\bf 53} (1946), 248--250.
%

\bibitem{erd60}
P.\ Erd\H os.
\newblock On sets of distances of $n$ points in Euclidean space.
\newblock \emph{Magyar Tudom\'anyos Akad\'emia Matematikai Kutat\'o Int\'ezet K\"ozlem\'enyei} {\bf 5} (1960), 165--169.
%

\bibitem{FPSSZ14}
J.\ Fox, J.\ Pach, A. Sheffer, A.\ Suk, and J.\ Zahl.
\newblock A semi-algebraic version of Zarankiewicz's problem.
\newblock \emph{J.\ Eur.\ Math.\ Soc.}, to appear.
%

\bibitem{Gross21}
E.\ Grosswald.
\newblock \emph{Representations of integers as sums of squares}.
\newblock Springer Science \& Business Media, 2012.
%

\bibitem{GK15}	
L.\ Guth and N.H.\ Katz.
\newblock On the Erd{\H o}s distinct distances problem in the plane.
\newblock {\em Annals Math.} {\bf 181} (2015), 155--190.
%

\bibitem{HR90}
T.\ Hagerup and C.\ R\"ub.
\newblock A guided tour of Chernoff bounds.
\newblock \emph{Information processing letters} {\bf 33} (1990), 305--308.
%

\bibitem{Hart00}
R.\ Hartshorne.
\newblock \emph{Geometry: Euclid and Beyond}.
\newblock Springer-Verlag, New York, 2000.
%

\bibitem{Iwaniec97}
H.\ Iwaniec.
\newblock \emph{Topics in classical automorphic forms}.
\newblock American Mathematical Soc., 1997.
%

\bibitem{KMSS12}
H.\ Kaplan, J.\ Matou\v sek, Z.\ Safernov\'a, and M.\ Sharir.
\newblock Unit distances in three dimensions.
\newblock \emph{Combinat.\ Probab.\ Comput.} {\bf 21} (2012), 597--610.
%

\bibitem{PS92}
J.\ Pach and M.\ Sharir.
\newblock Repeated angles in the plane and related problems.
\newblock \emph{J.\ Combinat.\ Theory Ser.\ A } {\bf 59} (1992), 12--22.
%

\bibitem{PS98}
J.\ Pach and M.\ Sharir.
\newblock On the number of incidences between points and curves.
\newblock \emph{Combinat.\ Probab.\ Comput.} {\bf 7} (1998), 121--127.
%

\bibitem{SchoenShk13}
T.\ Schoen and I.\ D.\ Shkredov.
\newblock Higher moments of convolutions.
\newblock \emph{J.\ Number Theor.} {\bf 133} (2013), 1693--1737.
%

\bibitem{SolySza}
J.\ Solymosi and E.\ Szab\'o.
\newblock unpublished manuscript.
%

\bibitem{Tao09}
T.\ Tao.
\newblock \emph{Poincar\'e's legacies: pages from year two of a mathematical blog}.
\newblock American Mathematical Soc., 2009.
%

\bibitem{TV06}
T.\ Tao and V.\ H.\ Vu.
\newblock \emph{Additive combinatorics}.
\newblock Cambridge University Press, 2006.
%

\bibitem{Zahl13}
J.\ Zahl.
\newblock An improved bound on the number of point-surface incidences in three dimensions.
\newblock \emph{Contrib. Discrete Math.} {\bf 8} (2013), 100--121.
%

\end{thebibliography}

%% AUTHOR: You can generate such a bibliography from a .bib file by
%% running pdflatex/bibtex/pdflatex/pdflatex and then pasting the .bbl file
%% between \begin{thebibliography} and \end{bibliography}

%%% AUTHOR: Include a short description of each author following the
%%% structure below. Use the same short tags used previously.
%%% Use \imageat{} and \imagedot{} instead of "@" and "." in
%%% email addresses-this replaces the symbols with graphics to avoid
%%% e-mail address harvesting from the .pdf file
\begin{dajauthors}
\begin{authorinfo}[pgom]
  Adam Sheffer\\
  California Institute of Technology\\
  Pasadena, CA\\
  USA\\
  \href{mailto:adamsh@gmail.com}{adamsh@gmail.com}
\end{authorinfo}
\end{dajauthors}

\end{document}